\documentclass[12pt,twoside,a4paper]{article}

\usepackage{a4,enumerate}
\usepackage[noadjust]{cite}
\usepackage{amssymb,amsmath,amsthm}
\usepackage{comment}
\usepackage{accents,calc}
\usepackage{xcolor}

\newcommand\mytitle{On the existence of distributional potentials} 
\newcommand\lhead{J. Voigt}
\newcommand\rhead{Existence of distributional potentials}
\swapnumbers
\numberwithin{equation}{section}

\newtheorem{theorem}{Theorem}[section]
\newtheorem{corollary}[theorem]{Corollary}
\newtheorem{proposition}[theorem]{Proposition}
\newtheorem{lemma}[theorem]{Lemma}
\theoremstyle{definition}

\newtheorem{remark}[theorem]{Remark}

 \mathchardef\ordinarycolon\mathcode`\:
  \mathcode`\:=\string"8000
  \begingroup \catcode`\:=\active
    \gdef:{\mathrel{\mathop\ordinarycolon}}
  \endgroup

\def\bigscpr(#1,#2){{\left(#1\nonscript \mskip2mu plus2mu \middle| \nonscript \mskip2mu
plus2mu#2\right)}}

\newcommand\ran{R}
\renewcommand\ran{\operatorname{\rm ran}}
\newcommand*\indic{\mathbf{1}}

\newcommand\la{\lambda}

\newcommand*\bogi{\u\i}
\renewcommand\tilde{\widetilde}
\renewcommand\hat{\widehat}

\newcommand{\spt}{\operatorname{spt}}

\newcommand\grad{\operatorname{grad}}

\renewcommand*\div{\operatorname{div}}
\newcommand*\rot{\operatorname{rot}}

\renewcommand\phi{\varphi}

\renewcommand\epsilon{\varepsilon}

\newcommand*\Hone{H^1}
\newcommand*\Honen{\Hone_0}
\newcommand*\Honens{\Hone_{0,\sigma}}
\newcommand*\Honesn{\Hone_{\sigma,0}}
\newcommand*\polar{{\setlength{\unitlength}{.1em}%
\begin{picture}(4.4,5)%
\put(2.2,2.8){\circle{3.1}}%
\end{picture}}}
\renewcommand*\polar{{\textstyle\circ}}

\newcommand{\R}{\mathbb{R}\nonscript\hskip.03em}
\newcommand{\N}{\mathbb{N}\nonscript\hskip.03em}

\newcommand{\K}{\mathbb{K}\nonscript\hskip.03em}

\newcommand\cD{\mathcal D}

\newcommand\dupa[2]{\left\langle #1, #2 \right\rangle}
\newcommand\dup[2]{\langle #1, #2 \rangle}
\newcommand\bdup[2]{\bigl\langle #1, #2 \bigr\rangle}

\newcommand\textint{{\textstyle\int}}

\def\formE(#1,#2){\sum_{e\in E}\int_{a_e}^{b_e} #1_e'(x)\ol{#2_e'(x)}\,dx}

\makeatletter
\let\qedhere@ams\qedhere
\def\qedhere{\@ifnextchar[{\@qedhere}{\qedhere@ams}}
\def\@qedhere[#1]{\tag*{\raisebox{-#1ex}{\qedhere@ams}}}
\makeatletter
\def\env@cases{%
  \let\@ifnextchar\new@ifnextchar
  \left\lbrace
  \def\arraystretch{1.1}%
  \array{@{\,}l@{\quad}l@{}}%
}
\renewcommand\section{\@startsection {section}{1}{\z@}%
                                     {-3.25ex \@plus -1ex \@minus -.2ex}%
                                     {1.5ex \@plus.2ex}%
                                     {\normalfont\large\bfseries}}
\makeatother

\newcommand\restrict{\vphantom f\mskip1mu\vrule\mskip2mu}
\newcommand\set[2]{\bigl\{#1{;}\;#2\bigr\}}

\newcommand\ol{\overline}

\newcommand\pd{\partial}
\newcommand\bsp{\mkern-5.5mu}

\newcommand\<{\mkern-1.5mu}

\newcommand\comp{{\textnormal c}}
\newcommand\Cci{{\displaystyle 
C_{\raise0.2ex\hbox{$\scriptstyle\comp$}}^\infty}}

\renewcommand\le{\leqslant}

\renewcommand\ge{\geqslant}

\newcommand\sse{\subseteq}

\newcommand\di{\mathclose{}\,\mathrm{d}}

\newcommand\slim{\mathop{\rm s\kern.08em\mbox{\rm -}lim}} 

\newcommand\abstracttext{\noindent
We present proofs for the existence of distributional potentials $F\in\cD'(\Omega)$ for distributional
vector fields $G\in\cD'(\Omega)^n$, i.e.~$\grad F=G$, where $\Omega$ is an open subset of $\R^n$. The hypothesis in these proofs is the compatibility condition $\partial_jG_k=\partial_kG_j$ 
for all $j,k\in\{1,\dots,n\}$, if $\Omega$ is simply connected, and a stronger condition in the general case.
A key tool in our treatment is 
the Bogovski{\bogi} formula, assigning vector fields $v\in\cD(\Omega)^n$ satisfying $\div v=\phi$ to functions $\phi\in\cD(\Omega)$ with $\int \phi(x)\di x=0$. 
The results are applied to properties of Hilbert spaces of functions occurring in the treatment of  the Stokes operator and the Navier--Stokes equations.
\vspace{8pt}

\noindent
MSC 2010: 46F10, 46E35
\vspace{2pt}

\noindent
Keywords: Distribution, Poincar\'e's lemma, de Rham's theorem, Bogovski{\bogi} formula, Stokes operator
}
\begin{document}
\title{\mytitle}

\author{J\"urgen Voigt}

\date{}

\maketitle

\begin{abstract}
\abstracttext
\end{abstract}

\section*{Introduction}
\label{intro}

The most elementary version of Poincar\'e's lemma is the statement that, given a $C^1$-vector field $v\colon\R^n\to\K^n$ satisfying $\pd_jv_k=\pd_kv_j$ for all
$j,k=1,\dots,n$, there exists a potential $p\in C^2(\R^n)$ for $v$, i.e.~$\grad p=v$. The main issue of this paper is to present proofs of
the following two distributional versions of this kind of existence theorem.

\begin{theorem}\label{thm-derham} (`de Rham style')
Let $\Omega\sse\R^n$ be open and connected. Let $G=(G_1,\ldots,G_n)\in\cD'(\Omega)^n$, and suppose that
\begin{equation}\label{hyp H'}
\langle G,\phi\rangle=0\qquad (\phi\in\cD(\Omega)^n,\ \div \phi=0).
\end{equation}
Then there exists a distribution $F\in\cD'(\Omega)$ such that $\grad F=G$. If $n\ge2$ and $G$ has compact support, then $F$ can be chosen with compact support.
\end{theorem}

\begin{theorem}\label{thm-poincare} (`Poincar\'e style')
Let $\Omega\sse\R^n$ be open and simply connected. Let $G\in\cD'(\Omega)^n$ be such that $\pd_jG_k=\pd_kG_j$ for all $j,k\in\{1,\dots,n\}$. Then there exists $F\in\cD'(\Omega)$ such that $\grad F = G$.
\end{theorem}

Proofs of these theorems have been provided by S.\,Mardare \cite[Theorems~4.1 and~2.1]{Mardare2008}. In his paper he explains that ``to this day, there is no proof, to the best knowledge of the author''\!, of Theorem~\ref{thm-poincare} in the existing literature;  
and concerning Theorem~\ref{thm-derham}, the proof given in \cite[\S\,22, Theorem~17']{deRham1984} is said to require ``an important prerequisite about chains and flows on differential manifolds''\!. For more motivation and background we refer to \cite{Mardare2008}. A proof of the local part of Theorem~\ref{thm-poincare}, on the basis of de Rham's regularisation, is given in \cite[Corollary~3.6(i)]{Marsden1968}.

The author was intrigued by these theorems as well
as motivated by related problems connected with the Stokes operator, and this led to proofs of the theorems which are quite different from those presented in \cite{Marsden1968,Mardare2008}. 
Let us comment briefly on some differences between Mardare's and our proofs of Theorem~\ref{thm-poincare}, which is Theorem~2.1 in \cite{Mardare2008}. Given $G\in\cD'(\Omega)^n$ as in Theorem~\ref{thm-poincare}, Mardare first carries out the local part of the proof on sets $\omega=\prod_{j=1}^n(a_j,b_j)$ by presenting -- following Schwartz \cite[chap.~II, \S\,6]{Schwartz1966} -- a formula for the solution $F\in\cD'(\omega)$. Our approach is to deal with the local part on bounded open sets that are star-shaped with respect to an open ball and to approximate $G$ by $C^\infty$-vector fields satisfying the compatibility conditions. In both proofs, the global part is achieved by a homotopy argument; our proof is arranged in a way to permit the application of the divergence theorem on the parameter set $[0,1]^2$ of the homotopy.

The proofs of Theorems~\ref{thm-derham} and~\ref{thm-poincare} are the content of Sections~\ref{sec-derham} and~\ref{sec-poincare}, respectively.
In the remaining part of the paper we treat applications of the theorems stated above to Sobolev space versions of the theorems. In these results the task is finding potentials in $L_2(\Omega)$ for suitable vector fields in $H^{-1}(\Omega)^n$. Here, the author was motivated, amongst others, by the recent paper \cite{AmroucheCiarletMardare2015}. 

\section{Proof of de Rham's theorem}
\label{sec-derham}

In the proof of Theorem~\ref{thm-derham} we will need the Bogovski{\bogi} operator which we discuss next. Let $\Omega\sse\R^n$ be an 
open set, $\Omega$ star-shaped with respect to every point of an open ball $B(x^0,r_0)\sse\Omega$. Let $\rho\in\Cci(\R^n)_+:=\set{\phi\in\Cci(\Omega)}{\phi\geqslant0}$ with $\spt\rho\sse B(x^0,r_0)$ and $\int\rho(x)\di x=1$. For $\phi\in\cD(\Omega)$ we define $B\phi\in\cD(\Omega)^n$ by
\[
B\phi(x) := \int_\Omega\phi(y)(x-y)\int_1^\infty\rho(y+r(x-y))r^{n-1}\di r \di y\qquad(x\in\Omega).
\]
It is not too difficult to show that indeed $B\phi$ belongs to $C^\infty(\Omega)^n$, and
\begin{equation}\label{eq-support}
 \spt B\phi\sse\set{\la z_1+(1-\la)z_2}{z_1\in\spt\phi,\ z_2\in\spt\rho,\ 0\le\la\le1},
\end{equation}
a compact subset of $\Omega$.

\begin{remark}\label{rem-bogovskii}
The linear mapping $B\colon \cD(\Omega)\to\cD(\Omega)^n$ is the \textbf{Bogovski{\bogi} operator}; we refer to \cite{Bogovskii1979} for its first appearance. It has the important property that 
$\int_\Omega\phi(x)\di x=0$ implies that $\div\< B\phi=\phi$; see Remark~\ref{rem-explanation-Bogovskii} below. 
This property will also be used in the less explicit version that for any $\phi\in\cD(\Omega)$ with $\int\<\phi(x)\di x=0$ there exists a vector field $\Phi\in\cD(\Omega)^n$ with the property 
$\div\Phi=\phi$.
(The reader should be aware of the fact that the Bogovski{\bogi} operator depends on the function~$\rho$; so the use of the definite article `the' might be somewhat misleading.)
\end{remark}

\begin{lemma}\label{lem-bogcont}
The operator $B\colon \cD(\Omega)\to\cD(\Omega)^n$ is continuous with respect to the standard topologies. 
\end{lemma}

\begin{proof}
The function $B\phi$ can be rewritten as
\[
B\phi(x)=\int_{z\in\R^n}\phi(x-z)\frac z{|z|^n}\int_0^\infty\rho\bigl(x+s\frac z{|z|}\bigr)(s+|z|)^{n-1}\di s\di z
\]
(where $\rho$ and $\phi$ are considered as functions in $\cD(\R^n)$). This shows that derivatives of $B\phi$ can be estimated by derivatives of $\phi$ and $\rho$ of the same and lower order. This fact together with the support property \eqref{eq-support} shows the assertion.
\end{proof}

\begin{proof}[Proof of Theorem~\ref{thm-derham}]
(i) In this part of the proof we suppose that $\Omega$ is star-shaped with respect to an open ball $B(x^0,r)\sse\Omega$. Let $\rho\in\Cci(\R^n)_+$, 
$\spt\rho\sse B(x^0,r)$ and $\int\rho(x)\di x=1$, and let $B\colon\cD(\Omega)\to\cD(\Omega)^n$ be the corresponding Bogovski{\bogi} operator. For 
$\phi\in\cD(\Omega)$ we define
\[
 \dup F\phi:= - \dup G{B\phi}.
\]
Then Lemma~\ref{lem-bogcont} implies that $F\in\cD'(\Omega)$.

In order to show that $\grad F=G$ let $\phi\in\cD(\Omega)$, $j\in\{1,\dots,n\}$. Then 
$\int_\Omega\pd_j\phi(x)\di x=0$, hence $\div(B\pd_j\phi-\phi e_j)=\pd_j\phi-\pd_j\phi =0$ (where $e_j$ denotes the $j$-th unit vector), and hypothesis \eqref{hyp H'} implies
\[
\dup {G_j} \phi =\dup G{\phi e_j}= \dup G {B\pd_j\phi} = - \dup F{\pd_j\phi} = \dup{\pd_jF}\phi.
\]
This shows that $G=\grad F$.

(ii) For the proof of the general case let $\bigl(B(x^k,r_k)\bigr)_{k\in N}$ be a countable covering of $\Omega$ by open balls contained in $\Omega$, with $1\in N\sse\N$. To ease notation we put $\Omega_k:=B(x^k,r_k)$ ($k\in N$). For each $k\in N$ let $\rho_k\in\Cci(\R^n)_+$ with $\spt\rho_k\sse \Omega_k$, 
$\int\rho_k(x)\di x=1$, and let $B_k\colon\cD(\Omega_k)\to\cD(\Omega_k)^n$ be the corresponding Bogovski{\bogi} operator. 

For $k\in N$ and some $c_k\in\K$ we define $F_k\in\cD'(\Omega_k)$ by
\[
 \langle F_k,\phi\rangle:=c_k\textint\phi(x)\di x - \dup G {B_k\phi}\qquad(\phi\in\cD(\Omega_k)).
\]
Then part (i) shows that $\grad F_k=G$ on $\Omega_k$. We are going to show that the constants $c_k$ can be chosen such that the family $(F_k)_{k\in N}$ of distributions is consistent.

For $k\in N$ we observe that there exists a function $\Phi_k\in\cD(\Omega)^n$ such that $\div\Phi_k=\rho_1-\rho_k$; this can be  seen by connecting $\Omega_1$ with $\Omega_k$ by a finite chain of consecutively intersecting open balls and applying Remark~\ref{rem-bogovskii} repeatedly. With this function we choose $c_k$ in such a way that $\dupa{F_k}{\rho_k}=\dupa G{\Phi_k}$; then
\begin{equation*}
 \dupa{F_k}\phi = \dupa{F_k}{\rho_k+\div B_k(\phi-\rho_k)}= \dupa G{\Phi_k-B_k(\phi-\rho_k)}
\end{equation*}
for all $\phi\in\cD(\Omega_k)$ with $\int\phi(x)\di x=1$.

In order to show that the family $(F_k)_{k\in N}$ is consistent, let $k,l\in N$ be such that $\Omega_k\cap \Omega_l\ne\varnothing$, and let $\phi\in \cD(\Omega)$ with $\spt\phi\sse \Omega_k\cap \Omega_l$,  $\int\phi(x)\di x=1$. Then
\[
 \dupa {F_k}\phi - \dupa {F_l}\phi =\dupa G{\Phi_k-B_k(\phi-\rho_k) - \Phi_l+B_l(\phi-\rho_l)}=0,
\]
because the divergence of the function to which $G$ is applied turns out to be $(\rho_1-\rho_k)-(\phi-\rho_k)-(\rho_1-\rho_l)+(\phi-\rho_l)=0$.
It follows that $\dupa {F_k}\phi = \dupa {F_l}\phi$ for all $\phi\in\cD(\Omega_k\cap\Omega_l)$.

We conclude that the family $(F_k)_{k\in N}$ composes to a distribution $F\in\cD'(\Omega)$ satisfying $\grad F=G$.

(iii) Now suppose that $n\ge2$ and that $G$ has compact support. There exists $\psi\in\cD(\Omega)$ with $\psi=1$ in a neighbourhood of the support of $G$. Then
\[
\dup{\widehat G}\phi:=\dup G{\psi\phi}\qquad(\phi\in\cD(\R^n)^n)
\]
defines an extension of $G$ to $\cD(\R^n)^n$, satisfying $\pd_j\widehat G_k=\pd_k\widehat G_j$ for all $j,k\in\{1,\dots,n\}$. According to part (ii) above, there exists $\widehat F\in\cD'(\R^n)$ such that $\grad\widehat F=\widehat G$. Denote by $\Omega_\infty$ the unbounded component of $\R^n\setminus\spt G$. Then $\grad \widehat F=0$ on the 
connected open set $\Omega_\infty$; hence there exists $c\in\K$ such that $\widehat F =c$
on $\Omega_\infty$. Then the restriction of $\widehat F -c$ to $\cD(\Omega)$ is as asserted.
\end{proof}

\section{Proof of the distributional version of Poincar\'e's lemma}
\label{sec-poincare}

We start with a property that will be needed in the proof of Theorem~\ref{thm-poincare}.

\begin{lemma}\label{lem-homotopy}
 Let $\Omega\sse\R^n$ be open. Let $v\in C^1(\Omega;\K^n)$ be a vector field satisfying
 \[
 \pd_jv_k=\pd_kv_j\qquad(j,k=1,\dots,n).
 \]
Let $\gamma,\tilde\gamma\colon[0,1]\to\Omega$ be $C^1$-paths with $\gamma(0)= \tilde\gamma(0)=:x^0$, $\gamma'(0)= \tilde\gamma'(0)=0$, $\gamma(1)= \tilde\gamma(1)=:x^1$, $\gamma'(1)= \tilde\gamma'(1)=0$.  
Suppose that $\Gamma\colon[0,1]^2\to\Omega$ is an FEP-homotopy between $\gamma$ and $\tilde\gamma$, i.e., $\Gamma$ is continuous, $\Gamma(\cdot,0)=\gamma$, $\Gamma(\cdot,1)=\tilde\gamma$,  $\Gamma(0,\cdot)=x^0$,  $\Gamma(1,\cdot)=x^1$. (FEP stands for `fixed end points'.) Then
\[
\int_\gamma v := \int_0^1v(\gamma(s))\cdot\gamma'(s)\di s=\int_0^1v(\tilde\gamma(s))\cdot\tilde\gamma'(s)\di s = \int_{\tilde\gamma} v.
\]
\end{lemma}

\begin{proof}
(i) In this step we suppose additionally that $\Gamma$ is twice continuously differentiable. We define the vector field $w\colon[0,1]^2\to \K^2$,
\[
 w(s,t):=\bigl(v(\Gamma(s,t))\cdot\frac\pd{\pd t}\Gamma(s,t),-v(\Gamma(s,t))\cdot\frac\pd{\pd s}\Gamma(s,t)\bigr).
\]
Then, using the hypothesis on $v$, one easily obtains $\div w=0$; hence the divergence theorem yields 
$ \int_{\pd[0,1]^2} w(s,t)\cdot\nu(s,t)\di\sigma(s,t) =0$,
where $\nu$ is the outer unit normal and $\sigma$ the surface measure. The integrals over the lines with $s=0$ and $s=1$ vanish, because 
$\frac\pd{\pd t}\Gamma(s,t)=0$ for $s=0$, $s=1$ and all $t\in[0,1]$. Hence the remaining integrals yield
\[
\int_0^1v(\Gamma(s,0))\cdot\frac\pd{\pd s}\Gamma(s,0)\di s - \int_0^1v(\Gamma(s,1))\cdot\frac\pd{\pd s}\Gamma(s,1)\di s = 0, 
\]
which is just the asserted equality.

(ii) In order to apply step (i) by smoothing the given homotopy $\Gamma$ we first `contract and extend' it as follows. We define the continuous function
$\alpha\colon[-1/4,5/4]\to[0,1]$ by $\alpha\restrict_{[-1/4,1/4]}:=0$, $\alpha\restrict_{[3/4,5/4]}:=1$, $\alpha\restrict_{[1/4,3/4]}$ affine linear,
and then we put $\tilde\Gamma\colon[-1/4,5/4]^2\to\Omega$, $\tilde\Gamma(s,t):=\Gamma(\alpha(s),\alpha(t))$. 
It is easy to see that then
\begin{equation}\label{eq-homot-1}
 \int_{\tilde\Gamma(\cdot,0)\restrict_{[0,1]}} v = \int_\gamma v\qquad\text{and}\qquad \int_{\tilde\Gamma(\cdot,1)\restrict_{[0,1]}} v = \int_{\tilde\gamma} v.
\end{equation}
Observe that
\[
 \tilde\Gamma(s,\cdot)\restrict_{[-1/4,1/4]}=\tilde\Gamma(s,0),\quad \tilde\Gamma(s,\cdot)\restrict_{[3/4,5/4]}=\tilde\Gamma(s,1)\qquad
 (s\in[-1/4,,5/4])
\]
and
\[
 \tilde\Gamma\restrict_{[-1/4,1/4]\times[-1/4,5/4]}=x^0,\quad \tilde\Gamma\restrict_{[3/4,5/4]\times[-1/4,5/4]}=x^1.
\]

Now let $(\rho_k)_{k\in\N}$ be a sequence in $\Cci(\R)_+$, $\spt\rho_k\sse[-1/k,1/k]$, $\int\rho_k(x)\di x=1$ for all $k\in\N$, and put
\[
 \tilde\Gamma_k(s,t):=\int_{-1/4}^{5/4}\int_{-1/4}^{5/4}\tilde\Gamma(s',t')\rho_k(s-s')\rho_k(t-t')\di s'\di t'
 \qquad(s,t\in[0,1],\ k\ge4).
\]
The properties mentioned above imply that $\tilde\Gamma_k(\cdot,j)=\tilde\Gamma(\cdot,j)*\rho_k$ and $\tilde\Gamma_k(j,\cdot)=\gamma(j)$ 
on $[0,1]$ for $j=0,1$ and all $k\ge4$. Using the uniform continuity of $\tilde\Gamma$ one shows that there exists $k_0\ge4$ such that 
$\tilde\Gamma_k([0,1]^2)\sse\Omega$ for all $k\ge k_0$.

Having established these properties, we can apply part (i) of the proof to conclude that
\begin{equation}\label{eq-homot-2}
 \int_{\tilde\Gamma_k(\cdot,0)} v = \int_{\tilde\Gamma_k(\cdot,1)} v \qquad (k\ge k_0). 
\end{equation}
From $\tilde\Gamma_k(\cdot,j)=\tilde\Gamma(\cdot,j)*\rho_k \to \tilde\Gamma(\cdot,j)$ in $C^1([0,1];\R^n)$ ($k\to\infty$) it follows that
\begin{equation}\label{eq-homot-3}
 \lim_{k\to\infty}\int_{\tilde\Gamma_k(\cdot,j)} v = \int_{\tilde\Gamma(\cdot,j)\restrict_{[0,1]}} v \qquad(j=0,1).
\end{equation}

Combining \eqref{eq-homot-2}, \eqref{eq-homot-3} and \eqref{eq-homot-1} we obtain the assertion of the lemma.
\end{proof}

\begin{proof}[Proof of Theorem~\ref{thm-poincare}]
(i) In this part we prove the theorem for the case that $\Omega$ is bounded and star-shaped with respect to an open ball $B(x^0,r)\sse\Omega$. Let
$\rho\in\Cci(\Omega)_+$ with $\spt\rho\sse B(x^0,r)$ and $\int\rho(x)\di x=1$, and let $B$ be the corresponding Bogovski{\bogi} operator. 
The transpose of $B\colon\cD(\Omega)\to\cD(\Omega)^n$ is the linear mapping $B^t\colon\cD'(\Omega)^n\to \cD'(\Omega)$, given by
\[
 \dup{B^tG}\phi :=\dup G{B\phi}\qquad(\phi\in\cD(\Omega),\ G\in\cD'(\Omega)^n),
\]
and $B^t$ is continuous with respect to the standard topologies. 

The restriction of $-B^t$ to $C^\infty(\Omega;\K^n)$ is the operator $A\colon C^\infty (\Omega;\K^n)\to C^\infty(\Omega)$, given by
\begin{equation}\label{equ-Bogovskii-adjoint}
 Av(x):=\int_\Omega\rho(y)\int_0^1v(sy+(1-s)x)\cdot(x-y)\di s\di y;
\end{equation}
see Remark~\ref{rem-explanation-Bogovskii}.
It has the property that $\grad Av=v$ for all 
\[
 v\in C^\infty_\gamma(\Omega;\K^n):=\set{u\in C^\infty(\Omega;\K^n)}{\pd_ju_k=\pd_ku_j\ (j,k=1,\dots,n)}.
\]

Now let $G\in\cD'(\Omega)^n$ with $\pd_jG_k=\pd_kG_j$ for all $j,k\in\{1,\dots,n\}$. We show that $G$ can be approximated by vector fields $v\in C^\infty_\gamma(\Omega;\K^n)$. 
Without loss of generality we assume that $x^0$ (the centre of the ball mentioned initially) is the origin. 
In a first step
we introduce distributions $G^\la$ for $\la\in(1,\infty)$, defined on $\la\Omega$ by
\[
\dup {G^\la}\phi := \bdup G{\phi\bigl(\tfrac1\la\,\cdot\,\bigr)}\qquad(\phi\in \cD(\la\Omega)^n).
\]
It is easy to see that $G^\la$ satisfies the compatibility condition and that for all $\phi\in\cD(\Omega)^n$ one obtains $\bdup{G^\la}\phi \to \dup G\phi$ as $\la\to 1$.

In order to approximate $G^\la$ from $C^\infty_\gamma(\Omega;\K^n)$ we fix $\la>1$ and observe that $\la\Omega$ is an open neighbourhood of the compact set $\ol\Omega$, due to the strict 
star-shapedness of $\Omega$. Let $(\rho_l)_{l\in\N}$ be a $\delta$-sequence in $\Cci(\R^n)_+$. Then, for large $l$, the function 
\[
 x\mapsto v^{\la,l}(x) := \bdup{G^\la}{\rho_l(\cdot-x)}=\bigl(\bdup{G_k^\la}{\rho_l(\cdot-x)}\bigr)_{k=1,\dots,n}
\]
is defined on $\Omega$ and belongs to $C^\infty(\Omega;\K^n)$. For all $j,k=1,\dots,n$ one obtains
\begin{align*}
\pd_{j,x}\bdup {G_k^\la}{\rho_l(\cdot-x)} &=\bdup{G_k^\la}{\pd_{j,x}\rho_l(\cdot-x)}= -\bdup{G_k^\la}{\pd_j\rho_l(\cdot-x)} =\bdup{\pd_jG_k^\la}{\rho_l(\cdot-x)}\\
&=\bdup{\pd_kG_j^\la}{\rho_l(\cdot-x)} = \cdots =\pd_{k,x}\bdup {G_j^\la}{\rho_l(\cdot-x)};
\end{align*}
hence $v^{\la,l}\in C^\infty_\gamma(\Omega;\K^n)$.

In order to keep notation simple we compute the convergence $v^{\la,l}\to G^\la$ in $\cD'(\Omega)^n$ componentwise (where we use the same symbol for the function $v^{\la,l}$ and for the regular distribution generated by this function). Thus, let $\phi\in\cD(\Omega)$, $k\in\{1,\dots,n\}$. Then
\begin{align*}
\bdup{v_k^{\la,l}}\phi &= \int v_k^{\la,l}(x)\phi(x)\di x=\int\bdup{G_k^{\la}}{\rho_l(\cdot-x)}\phi(x)\di x\\
&=\bdup{G_k^\la}{\int\rho_l(\cdot-x)\phi(x)\di x}= \bdup {G_k^\la}{\rho_l*\phi}.
\end{align*}
Now, $\rho_l*\phi\to\phi$ in $\cD(\la\Omega)$, hence $v_k^{\la,l}\to G_k^\la$ ($l\to\infty$) in the weak topology $\sigma(\cD'(\Omega),\cD(\Omega))$.

To complete this step we first recall from above that
\[
- \grad B^tv^{\la,l} = \grad Av^{\la,l}=v^{\la,l}.
\]
Letting $l\to\infty$ we obtain $-\grad B^tG^\la=G^\la$, and then letting $\la\to1$ we finally get $-\grad B^tG=G$. This shows that $F:=-B^tG$ has the asserted property.

(ii)
For the general case, let $(B(x^m,r_m))_{m\in N}$ be a countable covering of $\Omega$ by open balls, with $1\in N\sse\N$; for brevity we put 
$\Omega_m:= B(x^m,r_m)$ ($m\in N$). From part (i) we conclude that for each $m\in N$ there exists $F_m\in\cD'(\Omega_m)$ such that $\grad F_m=G$ on $\Omega_m$.

Now it remains to show that there exist constants $c_m\in\K$ such that $(c_m+F_m)_{m\in N}$ is a consistent family of distributions. Without restriction we assume that $N=\{1,\dots,m_0\}$ or $N=\N$, and that $\Omega_m\cap\bigcup_{k=1}^{m-1}\Omega_k\ne\varnothing$ for all $m\in N\setminus\{1\}$. Put $c_1:=0$. 

Let $m'\in N\setminus\{1\}$ and assume that $c_2,\dots,c_{m'-1}$ are found such that the family $(c_m+F_m)_{m=1,\dots,m'-1}$ is consistent; without restriction $c_2=\dots=c_{m'-1}=0$. Denote by $\hat F_{m'-1}\in\cD'(\bigcup_{m=1}^{m'-1}\Omega_m)$  the composed distribution satisfying $\widehat F_{m'-1}=F_m$ on $\Omega_m$ for all $m=1,\dots,m'-1$. There exists $m\in\{1,\dots,m'-1\}$ such that $\Omega_{m'}\cap \Omega_m\ne\varnothing$, and 
$\grad F_{m'}=G=\grad F_m$ on $\Omega_{m'}\cap\Omega_m$ implies that there exists $c_{m'}\in\K$ such that $c_{m'}+F_{m'}=F_m$ on $\Omega_{m'}\cap \Omega_m$; 
without restriction $c_{m'}=0$. In order to make sure that $F_{m'}$ is consistent with $\hat F_{m'-1}$ we have to show that $F_{m'}$ is consistent with $F_{\tilde m}$ for all 
$\tilde m\in\{1,\dots,m'-1\}\setminus\{m\}$. Thus, let $\tilde m\in\{1,\dots,m'-1\}\setminus\{m\}$ be such that $\Omega_{m'}\cap\Omega_{\tilde m}\ne\varnothing$.
Then as before there exists $\tilde c_{m'}\in\K$ such that $\tilde c_{m'}+F_{m'}=F_{\tilde m}$ on $\Omega_{m'}\cap \Omega_{\tilde m}$.
There exists a continuously differentiable path $\gamma\colon[0,1]\to \Omega_m\cup \Omega_{m'}$ from $x^m$ to $x^{m'}$ satisfying 
$\gamma'(0)=\gamma'(1)=0$, and there exists a continuously differentiable path $\tilde \gamma\colon[0,1]\to\bigcup_{k=1}^{m'} \Omega_k$ from $x^m$ to $x^{m'}$ satisfying $\tilde\gamma([0,1/2])\sse\bigcup_{k=1}^{m'-1}\Omega_k$, $\tilde\gamma(1/2)=x^{\tilde m}$, 
$\tilde\gamma([1/2,1])\sse \Omega_{\tilde m}\cup \Omega_{m'}$, and $\tilde\gamma'(0)=\tilde\gamma'(1)=0$.
The simple connectedness of $\Omega$ implies that there exists an FEP-homotopy $\Gamma\colon[0,1]^2\to\Omega$ between $\gamma$ and $\tilde\gamma$. 
As $\Gamma([0,1]^2)$ is compact, there exists $r>0$ such that $\Gamma([0,1]^2)+B(0,2r)\sse\Omega$. Let $\rho\in\Cci(\R^n)_+$, $\spt\rho\sse B(0,r)$, 
$\int\bsp\rho(x)\di x=1$, and put
\[
 v(x):=\dup G{\rho(\cdot-x)} \qquad(x\in\Gamma([0,1]^2)+B(0,r)).
\]
Then $v\in C^\infty(\Gamma([0,1]^2)+B(0,r);\K^n)$, and the
hypothesis on $G$ implies that  $\pd_jv_k=\pd_kv_j$ for all $j,k=1,\dots,n$. Applying Lemma~\ref{lem-homotopy} with the open set 
$\Gamma([0,1]^2)+B(0,r)$ we conclude that
\begin{equation}\label{eq-appl-homot}
 \int_\gamma\bsp v =\int_{\tilde\gamma}\bsp v
\end{equation}
The distributions 
$F_m$, $F_{m'}$ are consistent; call $\check F\in\cD'(\Omega_m\cup \Omega_{m'})$ the combined distribution. Then $\grad\check F=G$ on 
$\Omega_m\cup \Omega_{m'}$, and
\begin{align*}
\dup G{\rho(\cdot-\gamma(s))}\cdot\gamma'(s)&=\dup{\grad\check F}{\rho(\cdot-\gamma(s))}\cdot\gamma'(s)\\
&=-\dup{\check F}{\grad\rho(\cdot-\gamma(s))}\cdot\gamma'(s)
=\frac\di{\di s}\dup{\check F}{\rho(\cdot-\gamma(s)},
\end{align*}
\begin{align*}
\int_\gamma\bsp v &= \int_0^1\<\dup G{\rho(\cdot-\gamma(s))}\cdot\gamma'(s)\di t
=\dup{\check F}{\rho(\cdot-x^{m'})}-\dup{\check F}{\rho(\cdot-x^{m})}\\
&=\dup{F_{m'}}{\rho(\cdot-x^{m'})}-\dup{F_m}{\rho(\cdot-x^{m})}\\
&=\dup{F_{m'}}{\rho(\cdot-x^{m'})}-\dup{\hat F_{m'-1}}{\rho(\cdot-x^{m})}.
\end{align*}
By the same token,
\begin{align*}
\int_{\tilde\gamma}\bsp v &=\int_0^1\<\dup G{\rho(\cdot-\tilde\gamma(s))}\cdot\tilde\gamma'(s)\di s\\
&=\int_0^{1/2}\<\dup G{\rho(\cdot-\tilde\gamma(s))}\cdot\tilde\gamma'(s)\di s+ \int_{1/2}^1\<\dup G{\rho(\cdot-\tilde\gamma(s))}\cdot\tilde\gamma'(s)\di s\\
&=\dup{\hat F_{m'-1}}{\rho(\cdot-x^{\tilde m})}-\dup{\hat F_{m'-1}}{\rho(\cdot-x^{m})}\\
&\qquad\qquad\qquad+ \dup{\tilde c_{m'}+F_{m'}}{\rho(\cdot-x^{m'})}-\dup{F_{\tilde m}}{\rho(\cdot-x^{\tilde m})}\\
&=\tilde c_{m'}+\dup{F_{m'}}{\rho(\cdot-x^{m'})}-\dup{\hat F_{m'-1}}{\rho(\cdot-x^{m})}.
\end{align*}
From \eqref{eq-appl-homot} we obtain $\tilde c_{m'}=0$, and this shows that the family $(F_m)_{m\in\{1,\dots,m'\}}$ is consistent; hence by induction, the family $(F_m)_{m\in N}$ is consistent.
\end{proof}

\begin{remark}
We note that Theorem~\ref{thm-poincare} implies classical versions of Poincar\'e's lemma: If $\Omega\sse\R^n$ is open and simply connected, and $g\in C^l(\Omega;\K^n)$ (for some $l\in\N_0$) is a vector field satisfying
\[
 \pd_j g_k=\pd_k g_j\qquad(j,k=1,\dots,n),
\]
then there exists $f\in C^{l+1}(\Omega)$ with $\grad f=g$. Indeed, the distributional solution $F$ of $\grad F=g$ obtained by Theorem~\ref{thm-poincare} is automatically a regular distribution generated by a $C^{l+1}$-function.
\end{remark}

\begin{remark}\label{rem-explanation-Bogovskii} On the Bogovski{\bogi} operator $B\colon \Cci(\Omega)\to \Cci(\Omega;\K^n)$ and the operator $A\colon C^\infty(\Omega;\K^n)\to C^\infty(\Omega)$ from \eqref{equ-Bogovskii-adjoint}.

Let $\Omega$ and $\rho$ be as at the beginning of Section~\ref{sec-derham}. It is a standard exercise of calculus that $\grad Av=v$ for all $v\in C_\gamma^\infty(\Omega;\K^n)$. Substitution of variables yields $\int(Av)\phi\di x=-\int v\cdot B\phi\di x$ for all $v\in C^\infty(\Omega;\K^n)$, $\phi\in\Cci(\Omega)$.

Let $\phi,\psi\in\Cci(\Omega)$, $\int \phi\di x=0$. Note that $\grad\psi\in C_\gamma^\infty(\Omega;\K^n)$, therefore $\grad(A\grad\psi-\psi)=0$; hence $A\grad\psi-\psi$ is constant, $\int(A\grad\psi-\psi)\phi\di x=0$. This implies
\[
 \int\psi\div\< B\phi\di x=-\int\grad\psi\cdot B\phi\di x=\int (A\grad\psi) \phi\di x=\int\psi \phi\di x.
\]
As this equality holds for all $\psi\in\Cci(\Omega)$ one obtains $\div\< B\phi =\phi$.
\end{remark}

\section{The `coarse' and `simplified' versions of de Rahm's theorem}
\label{sec-derham-c-s}

In this section we treat the existence of potentials in a Hilbert space context. Let $\Omega\sse\R^n$ be a connected bounded open set.
We define the Sobolev spaces
\begin{align*}
 \Honens(\Omega;\K^n)&:=\set{u\in\Honen(\Omega;\K^n)}{\div u=0},\\
 \Honesn(\Omega;\K^n)&:=\ol{C^\infty_{\comp,\sigma}(\Omega;\protect{\K^n})}^{\Honen}, 
 \end{align*}
 where $C^\infty_{\comp,\sigma}(\Omega;\K^n):=\set{\phi\in\Cci(\Omega;\K^n)}{\div \phi=0}$. For a subspace $V\sse\Honen(\Omega;\K^n)$ we define the polar
 \[
  V^\polar:=\set{g\in H^{-1}(\Omega)^n}{\dup g \phi_{H^{-1},\Honen}=0\ (\phi\in V)},
 \]
 where $H^{-1}(\Omega)$ is the anti-dual space of $\Honen(\Omega)$ in the Gelfand triple $\Honen(\Omega)\sse L_2(\Omega)\sse H^{-1}(\Omega)$.
 
 We adopt the terminology in the title of this section from \cite{AmroucheCiarletMardare2015} (formerly used already in \cite[Chap.\;I, \S\,2]{GiraultRaviart1986}), where it is defined that the ``coarse version of de Rham's theorem'' holds (for $\Omega$) if
 \begin{align}\tag{$\mathrm H_0$}\label{propH}
  \text{for all }g\in\Honens(\Omega;\K^n)^\polar \text{ there exists }f\in L_2(\Omega)\text{ with }g=\grad f,
 \end{align}
and  the ``simplified version of de Rham's theorem'' holds if
 \begin{align}\tag{$\mathrm H_\comp$}\label{propH'}
  \text{for all }g\in C^\infty_{\comp,\sigma}(\Omega;\K^n)^\polar \text{ there exists }f\in L_2(\Omega)\text{ with }g=\grad f.
 \end{align}
Note that property \eqref{propH'} could have been stated equivalently with $\Honesn(\Omega;\K^n)$ instead of $C^\infty_{\comp,\sigma}(\Omega;\K^n)$, because these spaces have the same polar.
Also note that in both of these properties the existence of $f$ could have been stated with the additional property that 
$f\in L_2^0(\Omega):=\set{g\in L_2(\Omega)}{\int g(x)\di x=0}$. 

From $\Honens(\Omega;\K^n)^\polar\sse\Honesn(\Omega;\K^n)^\polar$ it is clear that property \eqref{propH'} implies \eqref{propH}. It is shown in \cite[Theorem~4.1]{AmroucheCiarletMardare2015} that for $\Omega$ with Lipschitz boundary, ``J.\,L.\;Lions' lemma'' together with \eqref{propH} implies \eqref{propH'}; see also Remark~\ref{rem-ACM}. 
We are going to show that \eqref{propH'} is `hereditary' (Proposition~\ref{prop-H'localised}), that \eqref{propH} holds and implies \eqref{propH'} for strictly
star-shaped sets (Remark~\ref{rem-bog}) and that \eqref{propH'} holds for sets with Lipschitz boundary (Corollary~\ref{cor-H'localised}).

\begin{remark}\label{rem-dense}
Let $\grad$ denote the operator $\grad\colon L_2(\Omega)\to H^{-1}(\Omega)^n$, $f\mapsto\grad f$. Then $\ran(\grad)\sse\Honens(\Omega;\K^n)^\polar\sse \Honesn(\Omega;\K^n)^\polar$ 
($=C^\infty_{\comp,\sigma}(\Omega;\K^n)^\polar$), and \eqref{propH} is equivalent to $\ran(\grad)=\Honens(\Omega;\K^n)^\polar$, whereas \eqref{propH'} is equivalent to $\ran(\grad)=\Honesn(\Omega;\K^n)^\polar$. This implies that \eqref{propH'} is equivalent to \eqref{propH} together with $\Honesn(\Omega;\K^n)=\Honens(\Omega;\K^n)$. Note that the last equality is equivalent to the denseness of $C^\infty_{\comp,\sigma}(\Omega;\K^n)$ in $\Honens(\Omega;\K^n)$.
\end{remark}

We start our investigation of these properties by a proof of the equivalence of property \eqref{propH'} and ``J.\,L.\;Lions' lemma''; see 
\cite[Theorem~4.1]{AmroucheCiarletMardare2015}.

\begin{theorem}\label{thm-H'-lions}
Let $\Omega\sse\R^n$ be a connected bounded open set. Then property \eqref{propH'} for $\Omega$ is equivalent to the property that for all 
$F\in\cD'(\Omega)$ with $\grad F\in H^{-1}(\Omega)^n$ there exists $f\in L_2(\Omega)$ with $F=f$ (as distributions).
\end{theorem}

\begin{proof}
For the proof of the necessity let $F\in\cD'(\Omega)$ such that $\grad F\in H^{-1}(\Omega)^n$. Then $\grad F\in C^\infty_{\comp,\sigma}(\Omega;\K^n)^\polar$; hence \eqref{propH'} implies that there exists 
$f\in L_2(\Omega)$ with $\grad f=\grad F$, and because $\Omega$ is connected there exists $c\in\K$ such that $F$ is generated by the function $f+c\in L_2(\Omega)$.

For the sufficiency let $g\in C^\infty_{\comp,\sigma}(\Omega;\K^n)^\polar$. Then Theorem~\ref{thm-derham} implies that there exists $F\in\cD'(\Omega)$ such that $\grad F=g\in H^{-1}(\Omega;\K^n)$; hence there exists $f\in L_2(\Omega)$ with $F=f$, $g=\grad F=\grad f$.
\end{proof}

The following proposition shows a hereditary property of \eqref{propH'}.

\begin{proposition}\label{prop-H'localised}
Let $\Omega\sse\R^n$ be a connected bounded open set, and assume that there exists a (finite) covering $(\Omega_j)_{j=1,\dots,m}$ of $\Omega$ by connected open sets $\Omega_j\sse \Omega$ with property \eqref{propH'}. Then property \eqref{propH'} holds for  $\Omega$. 
\end{proposition}

\begin{proof}
Let $g\in \Honesn(\Omega;\K^n)^\polar$. Then clearly 
$g^j:=g\restrict_{\Honen(\Omega_j;\K^n)}
\in\Honesn(\Omega_j;\K^n)^\polar$, and the hypothesis implies that for all $j=1,\dots,m$ there exists 
$f_j\in L_2(\Omega_j)$ such that $g^j=\grad f_j$.
Note that each function $f_j$ is only determined up to a constant, and we have to `glue together' suitable 
versions of these functions.

We apply Theorem~\ref{thm-derham} and obtain a distribution $F\in\cD'(\Omega)$ such that $g=\grad F$.
This implies that for each $j\in\{1,\dots,m\}$ one has $\grad F=\grad f_j$ on~$\Omega_j$, and because $\Omega_j$ is connected, there
exists $c_j\in\K$ such that $F=f_j+c_j\indic_{\Omega_j}$ on~$\Omega_j$. This implies that the family $(f_j+c_j\indic_{\Omega_j})_{j=1,\dots,m}$ of $L_2$-functions
is consistent; hence there exists $f\in L_2(\Omega)$ with $f\restrict_{\Omega_j}=f_j+c_j\indic_{\Omega_j}$ for all $j=1,\dots,m$, $\grad f=\grad F=g$.
\end{proof}

Next we turn to more concrete sufficient conditions for the validity of \eqref{propH'}. The fundamental source of these conditions is an important fact from \cite[Lemma~1]{Bogovskii1979}, stated in part~(a) of the following remark.

\begin{remark}\label{rem-bog}
Let $\Omega\sse\R^n$ be a bounded open set which is star-shaped with respect to an open ball $B(x^0,r)$.

(a) Then the corresponding Bogovski{\bogi} operator $B\colon\Cci(\Omega)\to\Cci(\Omega;\K^n)$ possesses a (unique) bounded linear extension
$B\colon L_2(\Omega)\to\Honen(\Omega;\K^n)$. (The paper \cite{Bogovskii1979} does not contain a proof of this fact, but rather the hint that for the proof one should 
use~\cite{CalderonZygmund1956}; this hint has been executed in \cite[Theorem~2.4]{BorchersSohr1990}, \cite[Section~III.3]{Galdi2011}.) From $\div\< B\phi=\phi$ for $\phi\in\Cci\cap L_2^0(\Omega)$ it follows that $\div\< Bf=f$ for all $f\in L_2^0(\Omega)$. 

(b) It is easy to see that the bounded operators $\div\colon\Honen(\Omega;\K^n)\to L_2(\Omega)$ and $\grad\colon L_2(\Omega)\to H^{-1}(\Omega)^n$ are 
negative adjoints of each other, and therefore $\ran(\grad)^\polar=\ker(\div)$ -- a well-established relation for operators on Hilbert spaces -- implies 
$\overline{\ran(\grad)}=\ker(\div)^\polar$. As $\ran(\div)= L_2^0(\Omega)$ is closed, the closed range theorem implies that $\ran(\grad)$ is closed; hence
$\ran(\grad)=\ker(\div)^\polar=\Honens(\Omega;\K^n)^\polar$. This shows that \eqref{propH} is satisfied.

(c) We now show that even \eqref{propH'} is satisfied; this follows if we show that $C^\infty_{\comp,\sigma}(\Omega;\K^n)$ is dense in 
$\Honens(\Omega;\K^n)$. Without restriction $x^0=0$. Let $\phi\in \Honens(\Omega;\K^n)$. Then the extension of $\phi$ to $\R^n$ by zero belongs to $\Honens(\R^n;\K^n)$; we keep the notation $\phi$ for the extension. For $\lambda\in(1,\infty)$ we note that $\phi(\lambda\,\cdot)\in\Honens(\Omega;\K^n)$ has compact support
contained in $\lambda^{-1}\overline\Omega\sse\Omega$. Regularisation shows that $\phi(\lambda\,\cdot)$ is approximated by functions in $C^\infty_{\comp,\sigma}(\Omega;\K^n)$; hence $\phi(\lambda\,\cdot)\in\Honesn(\Omega;\K^n)$. Taking $\lambda\to1$ we then obtain $\phi\in\Honesn(\Omega;\K^n)$. (This reasoning has also been used in \cite[proof of Lemma~9, p.\,31]{Tartar1978}.)
\end{remark}

\begin{corollary}\label{cor-H'localised}
Let $\Omega\sse\R^n$ be a connected bounded open set, and assume that there exists a covering $(\Omega_j)_{j=1,\dots,m}$ of $\Omega$ by open sets 
$\Omega_j\sse \Omega$, where each $\Omega_j$ is 
star-shaped with respect to an open ball. Then property \eqref{propH'} holds for $\Omega$.

These assertions apply to any connected bounded open set with Lipschitz boundary.
\end{corollary}

\begin{proof}
From Remark~\ref{rem-bog}(c) we know that the sets $\Omega_j$ have property \eqref{propH'}. Hence Proposition~\ref{prop-H'localised} implies the assertion.

If $\Omega$ has Lipschitz boundary, then it is not difficult to
see that for all $x\in\ol{\Omega}$ there exists an open
neighbourhood $U_x$ such that $U_x\cap\Omega$ is star-shaped with respect to
the points of a ball in $U_x\cap\Omega$. This is obvious for $x\in\Omega$, and
for $x\in\partial\Omega$ it results from the Lipschitz property of $\partial\Omega$. The compactness of
$\ol{\Omega}$ implies that there exists a finite open covering
$(\Omega_j)_{j=1,\ldots,m}$ of $\Omega$ as required in the first part of the corollary. (See also \cite[Theorem~2.1]{AmroucheCiarletMardare2015}.)
\end{proof}

\begin{remark}\label{rem-ACM}
In \cite[Theorem~3.1]{AmroucheCiarletMardare2015} an `equivalence result' is established between various classical results, e.g.\ the (classical) J.\,L.\;Lions lemma, the J.\;Ne\v cas inequality, property \eqref{propH} and the property that $\div\colon \Honens(\Omega;\K^n)\to  L_2^0(\Omega)$ is surjective, for bounded open sets with Lipschitz boundary. Each of these results by itself is of a notoriously high technical level, and in \cite[Section~2]{AmroucheCiarletMardare2015} several sources for the results are mentioned.

A treatment of part of these properties, for the case that $\Omega$ has $C^1$-boundary, can be found in \cite[pp.~26--32]{Tartar1978}.
Related problems are treated in a more general setting in \cite{AmroucheGiraud1994}. For further information
we also mention \cite{Ciarlet2013, GiraultRaviart1986, Temam1977}.
\end{remark}

\section{A `weak version' of Poincar\'e's lemma}
\label{sec-wvPlemma}
 
In the previous section we have seen how the distributional version of de Rham's theorem can be applied in the $L_2$-context. Here we present an application of the distributional version of 
Poincar\'e's lemma. 

\begin{theorem}\label{thm-wvPlemma}
Let $\Omega\sse\R^n$ be a simply connected bounded open set satisfying property \eqref{propH'}. Let $g\in H^{-1}(\Omega)^n$ be such that $\pd_jg_k=\pd_kg_j$ for all $j,k=1,\dots,n$ (with derivatives in the sense of distributions).
Then there exists $f\in L_2(\Omega)$ such that $\grad f= g$.
\end{theorem}

\begin{proof}
Theorem~\ref{thm-poincare} implies that there exists $F\in\cD'(\Omega)$ such that $\grad F=g$. This shows that for all 
$\phi\in C^\infty_{\comp,\sigma}(\Omega;\K^n)$ one obtains
\[
 \dup g\phi_{H^{-1},\Honen} =\dup{\grad F}\phi=-\dup F{\div\phi} =0,
\]
i.e.\ $g\in C^\infty_{\comp,\sigma}(\Omega;\K^n)^\polar$. Now property \eqref{propH'} implies that there exists $f\in L_2(\Omega)$ such that $\grad f=g$.
\end{proof}

Recall from Corollary~\ref{cor-H'localised} that Theorem~\ref{thm-wvPlemma} applies to the case that $\Omega$ is a simply connected bounded open set with Lipschitz boundary. 
For this case, proofs of Theorem~\ref{thm-wvPlemma} are contained in \cite[proof of Theorem~2.1]{Kesavan2005}, \cite[Proof of Theorem~4.1]{AmroucheCiarletMardare2015}, without the use of
our Theorem~\ref{thm-poincare}.
\bigskip

\textbf{Acknowledgement.}
The author is grateful to M.\ Kunzinger (Vienna) for pointing out reference \cite{Marsden1968}.

{
}
\bigskip

\noindent
J\"urgen Voigt\\
Technische Universit\"at Dresden\\
Fakult\"at Mathematik\\
01062 Dresden, Germany\\
{\tt 
juer\rlap{\textcolor{white}{xxxxx}}gen.vo\rlap{\textcolor{white}{yyyyyyyyyy}}%
igt@tu-dr\rlap{\textcolor{white}{%
zzzzzzzzz}}esden.de}

\end{document}